\documentclass[12pt]{amsart}
\usepackage{amsfonts,amsmath,amsthm,amssymb}
\usepackage{latexsym}
\usepackage{enumerate}
\usepackage{graphics}
\newtheorem{theorem}{Theorem}[section]
\newtheorem{corollary}[theorem]{Corollary}
\newtheorem{construction}[theorem]{Construction}
\newtheorem{lemma}[theorem]{Lemma}

\newtheorem{conjecture}[theorem]{Conjecture}

\newtheorem{proposition}[theorem]{Proposition}

\newtheorem{definition}[theorem]{Definition}
\theoremstyle{definition}

\theoremstyle{remark}
\newtheorem{rem}[theorem]{Remark}
\theoremstyle{remark}

\newcommand{\beql}[1]{\begin{equation}\label{#1}}
\newcommand{\eeq}{\end{equation}}

\begin{document}

\title[Complex Hadamard matrices of order $6$]{Complex Hadamard matrices of order $6$:\\ a four-parameter family}

\author{Ferenc Sz\"oll\H{o}si}

\date{August, 2010., Preprint}

\address{Ferenc Sz\"oll\H{o}si: Department of Mathematics and its Applications, Central European University, H-1051, N\'ador u. 9, Budapest, Hungary.}\email{szoferi@gmail.com}

\thanks{This work was supported by the Hungarian National Research Fund OTKA K-77748}

\dedicatory{Dedicated to Professor Uffe Haagerup on the occasion of his $60$th birthday}

\begin{abstract}
In this paper we construct a new, previously unknown four-parameter family of complex Hadamard matrices of order $6$, the entries of which are described by algebraic functions of roots of various sextic polynomials. We conjecture that the new, generic family $G_6^{(4)}$ together with Karlsson's degenerate family $K_6^{(3)}$ and Tao's spectral matrix $S_6^{(0)}$ form an exhaustive list of complex Hadamard matrices of order $6$. Such a complete characterization might finally lead to the solution of the famous MUB-$6$ problem.
\end{abstract}

\maketitle

{\bf 2010 Mathematics Subject Classification.} Primary 05B20, secondary 46L10.
	
{\bf Keywords and phrases.} {\it Complex Hadamard matrices, mutually unbiased bases, MUB}

\section{Introduction}\label{sec:intro}
Complex Hadamard matrices form an important family of orthogonal arrays with the additional unimodularity constraint imposed on their entries. These matrices obey the algebraic identity $HH^\ast=n I_n$ where $\ast$ stands for the Hermitean transpose, and $I_n$ is the identity matrix of order $n$. They appear in various branches of mathematics frequently, including linear algebra \cite{godsil}, coding- and operator theory \cite{haagerup,popa} and harmonic analysis \cite{KM,tao}. They play an important r\^ole in quantum optics, high-energy physics, and they are one of the key ingredients to quantum teleportation- and dense coding schemes \cite{wer} and mutually unbiased bases (MUBs) \cite{bengtsson}. An example of complex Hadamard matrices is the Fourier matrix $F_n$, well-known to exists for all $n$. It is natural to ask how does a ``typical'' complex Hadamard matrix of order $n$ look like, and a satisfying answer to this question can be given provided we have a complete characterization of Hadamard matrices of order $n$ at our disposal. These types of problems, however, are notoriously difficult even for small $n$. Naturally, one is interested in the essentially different matrices only, and we identify two matrices $H$ and $K$ and say that they are equivalent if $H=P_1D_1KD_2P_2$ for some unitary diagonal matrices $D_1, D_2$ and permutational matrices $P_1,P_2$. Recall that a complex Hadamard matrix is dephased, if all entries in its first row and column are equal to $1$. While studying and classifying real Hadamard matrices is naturally a discrete, finite problem which can be handled by deep algebraic methods and sophisticated computer programs to some extent, the complex case, however, behaves essentially different. In particular, due to the appearance of various parametric families one cannot hope for a finite list of inequivalent matrices, but rather for a finite list of constructions, each of them leading to an infinite family of complex Hadamard matrices.

The complete classification of complex Hadamard matrices is available up to order $n=5$ only. It is trivial that $F_n$ is the unique complex Hadamard matrix for orders $n\leq 3$. The case $n=4$ is still elementary, and it was shown by Craigen that all complex Hadamard matrices of order $4$ belong to an infinite, continuous one-parameter family \cite{craigen}. In order $5$ we have uniqueness again, a result which is absolutely non-trivial already. In particular, Lov\'asz was the first who showed \cite{L} that $F_5$ is the only circulant complex Hadamard matrix in this order, and a decade later Haagerup managed to prove the uniqueness of $F_5$ by discovering an algebraic identity (cf.\ formula \eqref{HEQ}) relating the matrix entries in a surprising way \cite{haagerup}.

In order $6$ various one- \cite{BN,dita,MSZ,Z}, two- \cite{K3,star} and three-parameter families \cite{K1,K2} have been constructed recently and it is conjectured that these are part of a more general, four-parameter family of complex Hadamard matrices, yet to be discovered \cite{bengtsson}. This conjecture is supported by overwhelming numerical evidence \cite{skinner}, however so far only a fairly small subset of it was described by closed analytic formul\ae, including an isolated matrix $S_6^{(0)}$ and a three-parameter matrix $K_6^{(3)}$ \cite{K1,K2}.

The reason why the $6\times 6$ case received significant attention in the past couple of years is the fact that complex Hadamard matrices are closely related to mutually unbiased bases. Recall that two orthonormal bases of $\mathbb{C}^d$, $\mathcal{B}_1$ and $\mathcal{B}_2$ are unbiased if for every $e\in \mathcal{B}_1$, $f\in \mathcal{B}_2$ we have $\left|\left\langle e,f\right\rangle\right|^2=1/d$. A family of orthonormal bases is said to be mutually unbiased if every two of them are unbiased. The famous MUB-$6$ problem asks for the maximal number of mutually unbiased bases in $\mathbb{C}^6$. On the one hand this number is at least $3$, as there exists various infinite families of triplets of MUBs in this order \cite{Ph,star,Z}, on the other hand it is well-known that it cannot be larger than $7$ (cf.\ the references of \cite{bengtsson}). In fact, it is conjectured that a triplet is the best one can come up with in dimension $6$ \cite{Z}. The connection between MUBs and Hadamard matrices of order $6$ has been exploited in \cite{Ph} very recently, where a discretization scheme was offered to attack the problem and it was proved by means of computers, however, in a mathematically rigorous way, that the members of the two-parameter Fourier family $F_6^{(2)}(a,b)$ and its transpose cannot belong to a configuration of $7$ MUBs containing the standard basis in dimension $6$. One reasonable hope to finally settle the MUB-$6$ problem is to give a complete characterization of complex Hadamard matrices of order $6$ and apply the same technique to them.

The goal of this paper is to propose a general framework towards the complete classification of complex Hadamard matrices of order $6$. In particular, by characterizing the orthogonal triplets of rows in complex Hadamard matrices we generalize an observation of Haagerup \cite{haagerup} to obtain a new algebraic identity relating the matrix entries in an unexpected way. This is an essentially new tool to study complex Hadamard matrices of small orders, and one of the main achievements of this paper. We apply this result to obtain complex Hadamard matrices, moreover we conjecture that the the construction we present here reflects the true nature of complex Hadamard matrices of order $6$. It has the following three features: Firstly, it is general in contrast with the earlier attempts where always some additional extra structure was imposed on the matrices including self-adjointness \cite{BN}, symmetry \cite{MSZ}, circulant block structure \cite{star} or $H_2$-reducibility \cite{K1}. Secondly, it has $4$ degrees of freedom and thirdly all the entries of the obtained matrices can be described by algebraic functions of roots of various sextic polynomials. This suggests on the one hand the existence of a four-parameter family of complex Hadamard matrices of order $6$ and reminds us on the other hand the fact that the desired algebraical description where the entries are expressed by radicals might not be possible at all. However, from the applicational point of view, and in particular, to utilize the computer-aided attack of \cite{Ph} to the MUB-$6$ problem we shall need these matrices numerically anyway. 

The outline of the paper is as follows. In Section $2$ we briefly discuss the main ideas of the construction to motivate the various auxiliary results we prove there. The excited reader might want to skip this section at first, and jump right ahead to Section $3$, where the construction of the new family is presented from a high-level perspective. In Section $4$ we analyze the construction thoroughly.

\section{Preliminary results}
In this section we present the ingredients necessary to construct our new family of complex Hadamard matrices of order $6$, including a characterization of the mutual orthogonality of three rows in Hadamard matrices (cf.\ Theorem \ref{HTIO}). First, however, we would like to motivate these efforts by describing the main ideas of the construction.

We start with a submatrix
\beql{E}
E(a,b,c,d):=\left[\begin{array}{ccc}
1 & 1 & 1\\
1 & a & b\\
1 & c & d\\
\end{array}\right]
\eeq
and attempt to embed it into a complex Hadamard matrix of order $6$
\beql{G6}
G_6^{(4)}(a,b,c,d):=\left[\begin{array}{ccc|ccc}
1 & 1 & 1 & 1 & 1 & 1\\
1 & a & b & e & s_1 & s_2\\
1 & c & d & f & s_3 & s_4\\
\hline
1 & g & h & \ast & \ast & \ast\\
1 & t_1 & t_3 & \ast & \ast & \ast\\
1 & t_2 & t_4 & \ast & \ast & \ast\\
\end{array}\right]\equiv\left[\begin{array}{cc}
E & B\\
C & D\\
\end{array}\right].
\eeq
with $3\times 3$ blocks $E,B,C$ and $D$ in two steps, as follows. First we construct the submatrices $B$ and $C$ featuring unimodular entries to obtain three orthogonal rows and columns of $G_6$. Secondly we find the unique lower right submatrix $D$ to get a unitary matrix. Should the entries of this matrix become unimodular, we have found a complex Hadamard matrix. We conjecture that the submatrix $E$ can be chosen, up to equivalence, in a way that there will be only finitely many candidates for the blocks $B$ and $C$ and therefore we can ultimately decide whether the submatrix $E$ can be embedded into a complex Hadamard matrix. The resulting matrix $G_6$ can be thought as the ``Hadamard dilation'' of the operator $E$.

We shall heavily use the following through the paper without any further comment: the conjugate of a complex number of modulus $1$ is its reciprocal, and hence the conjugate of a multivariate polynomial with real coefficients depending on indeterminates of modulus $1$ is just the polynomial formed by entrywise reciprocal of the aforementioned indeterminates. We computed various Gr\"obner bases \cite{GB} in this paper with the aid of Mathematica. The reader is advised to use a computer algebra system for bookkeeping purposes and consult \cite{karol} for the standard notations for well-known complex Hadamard matrices such as $S_6^{(0)}, K_6^{(3)}$, etc.

We begin with recalling two elementary results from the existing literature.

\begin{lemma}\label{L1}
Suppose that we have a partial row $(1,a,b,e,\ast,\ast)$ composed from unimodular entries. Then one can specify some unimodular numbers $s_1$ and $s_2$ in place of the unknown numbers $\ast$ to make this row orthogonal to $(1,1,1,1,1,1)$ if and only if
\beql{eq1}
|1+a+b+e|\leq 2.
\eeq
\end{lemma}
\begin{proof}
To ensure orthogonality, we need to have $1+a+b+e+s_1+s_2=0$ from which it follows that $|1+a+b+e|=|s_1+s_2|\leq 2$. It is easily seen geometrically, that in this case we can define the unimodular numbers required.
\end{proof}
The missing coordinates featuring in Lemma \ref{L1}, $s_1$ and $s_2$, can be obtained algebraically through the well-known
\begin{lemma}[(Decomposition formula, \cite{MSZ})]
Suppose that the rows $(1,1,1,1,1,1)$ and $(1,a,b,e,s_1,s_2)$ containing unimodular entries are orthogonal. Let us denote by $\Sigma:=1+a+b+e$, and suppose that $0<|\Sigma|\leq 2$. Then
\beql{s12}
s_{1,2}=-\frac{\Sigma}{2}\pm\mathbf{i}\frac{\Sigma}{|\Sigma|}\sqrt{1-\frac{|\Sigma|^2}{4}}.
\eeq
If $\Sigma=0$ then $s_1$ is independent from $a,b,e$ but $s_2=-s_1$.
\end{lemma}
\begin{proof}
Clearly $s_1$ and $s_2$ are the unimodular numbers with $s_1+s_2=-\Sigma$.
\end{proof}
Now we proceed by investigating the orthogonality of triplets of rows. In order to do this, the following is a crucial
\begin{definition}[(Haagerup polynomial)]
The Haagerup polynomial $\mathcal{H}$ associated to the rows $(1,1,1,1,1,1), (1,a,b,e,\ast,\ast)$ and $(1,c,d,f,\ast,\ast)$ of a complex Hadamard matrix read
\[\mathcal{H}(a,b,c,d,e,f):=(1+a+b+e)(1+\overline{c}+\overline{d}+\overline{f})(1+c\overline{a}+d\overline{b}+f\overline{e}).\]
\end{definition}
The first result of ours is the following
\begin{theorem}\label{HTIO}
Suppose that we have the partial rows $(1,a,b,e,\ast,\ast)$ and $(1,c,d,f,\ast,\ast)$, composed from unimodular entries. Then one can specify some unimodular numbers $s_1,s_2,s_3$ and $s_4$ in place of the unknown numbers $\ast$ to make these rows together with $(1,1,1,1,1,1)$ mutually orthogonal if and only if
\beql{C11}
\mathcal{H}(a,b,c,d,e,f)=4-|1+a+b+e|^2-|1+c+d+f|^2-|1+c\overline{a}+d\overline{b}+f\overline{e}|^2
\eeq
with
\beql{NI}
|\mathcal{H}(a,b,c,d,e,f)|\leq 8.
\eeq
\end{theorem}
\begin{proof}
First we start by proving that \eqref{C11} holds. To do this, we utilize Haagerup's idea \cite{haagerup} as follows: by pairwise orthogonality, we find that
\[1+a+b+e=-s_1-s_2\]
\[1+\overline{c}+\overline{d}+\overline{f}=-\overline{s}_3-\overline{s}_4\]
\[1+c\overline{a}+d\overline{b}+f\overline{e}=-s_3\overline{s}_1-s_4\overline{s}_2\]
Now, by multiplying these three equations together we find that \eqref{NI} follows and further
\[\mathcal{H}=-(s_1+s_2)(\overline{s}_3+\overline{s}_4)(s_3\overline{s}_1+s_4\overline{s}_2)=-(s_1\overline{s}_3+s_2\overline{s}_4+s_1\overline{s}_4+s_2\overline{s}_3)(s_3\overline{s}_1+s_4\overline{s}_2)=\]
\[-|s_1\overline{s}_3+s_2\overline{s}_4|^2-(s_1\overline{s}_4+s_2\overline{s}_3)(s_3\overline{s}_1+s_4\overline{s}_2)=-|s_1\overline{s}_3+s_2\overline{s}_4|^2-2\Re(s_1\overline{s}_2+s_3\overline{s}_4).\]
To conclude the proof, we need to show that
\[2\Re(s_1\overline{s}_2+s_3\overline{s}_4)=|1+a+b+e|^2+|1+c+d+f|^2-4\]
holds, however this follows from the Decomposition formula easily.

To see the converse direction, we need to show that \eqref{C11} essentially encodes orthogonality. Let us use the notations $\Sigma:=1+a+b+e$, $\Delta:=1+c+d+f$, $\Psi:=1+c\overline{a}+d\overline{b}+f\overline{e}$. With this notation condition \eqref{C11} boils down to
\beql{SH}
\mathcal{H}=\Sigma\overline{\Delta}\Psi=4-|\Sigma|^2-|\Delta|^2-|\Psi|^2.
\eeq
Clearly, if $|\Sigma|\leq 2$ and $|\Delta|\leq 2$ hold, then by the decomposition formula we can find $s_1, s_2$, $s_3$ and $s_4$ to ensure orthogonality to row $(1,1,1,1,1,1)$. Now observe, that the mutual orthogonality of rows $(1,a,b,e,s_1,s_2)$ and $(1,c,d,f,s_3,s_4)$ reads
\beql{ORT}
\Psi+s_3\overline{s}_1+s_4\overline{s}_2=0.
\eeq
Suppose first that we have the trivial case $\Sigma=\Delta=0$. Then, by the decomposition formula we have $s_2=-s_1$ and $s_4=-s_3$, and \eqref{SH} implies that $|\Psi|=2$. Therefore, if we set the unimodular number $s_3:=-\Psi s_1/2$ the orthogonality equation \eqref{ORT} is fulfilled.

Suppose secondly, that we have $\Delta=0$, but $\Sigma\neq 0$. Then we have $s_4=-s_3$, and from \eqref{SH} it follows that $|\Sigma|\leq 2$, and in particular
\beql{PSI}
|\Psi|=\sqrt{4-|\Sigma|^2}.
\eeq
Now we can use the Decomposition formula to find out the values of $s_1$ and $s_2$ and the orthogonality equation \eqref{ORT} becomes
\beql{ORT2}
\Psi+s_3\left(-2\mathbf{i}\frac{\overline{\Sigma}}{|\Sigma|}\sqrt{1-\frac{|\Sigma|^2}{4}}\right)=0.
\eeq
This holds, independently of $s_3$, if $|\Sigma|=2$, as by \eqref{PSI} $\Psi=0$ follows. Otherwise, set the unimodular number
\[s_3:=-\mathbf{i}\frac{\Sigma\Psi}{|\Sigma||\Psi|}\]
to ensure the orthogonality through \eqref{ORT2}.

Finally, let us suppose that $\Sigma\neq 0$ and $\Delta\neq 0$. Now observe that in this case the value of $\Psi$ needed for formula \eqref{ORT} can be calculated through \eqref{SH}. The other ingredient, namely the value of $s_3\overline{s}_1+s_4\overline{s}_2$ can be established through the Decomposition formula, once we derive the required bounds $|\Sigma|\leq 2$ and $|\Delta|\leq 2$. Depending on the value of $\mathcal{H}$, we treat several cases differently.

CASE 1: Suppose that $-|\Psi|^2\leq \mathcal{H}$. This implies, by formula \eqref{SH}, that $|\Sigma|^2+|\Delta|^2\leq 4$, and in particular $|\Sigma|\leq 2$, $|\Delta|\leq 2$ hold. Next we calculate $|\Psi|$ from \eqref{SH}.

Suppose first that $\mathcal{H}\geq 0$. Hence, after taking absolute values, \eqref{SH} becomes
\[|\Psi|^2+|\Sigma||\Delta||\Psi|+|\Sigma|^2+|\Delta|^2-4=0,\]
and the only non-negative root is
\beql{ps1}
|\Psi|=\frac{-|\Sigma||\Delta|+\sqrt{(4-|\Sigma|^2)(4-|\Delta|^2)}}{2}.
\eeq

Now suppose that $-|\Psi|^2\leq \mathcal{H}<0$. Hence, after taking absolute values, \eqref{SH} becomes
\beql{NEG}
|\Psi|^2-|\Sigma||\Delta||\Psi|+|\Sigma|^2+|\Delta|^2-4=0,
\eeq
and we find that the only non-negative root we have under the assumption $|\Sigma|^2+|\Delta|^2\leq 4$ reads
\beql{ps2}
|\Psi|=\frac{|\Sigma||\Delta|+\sqrt{(4-|\Sigma|^2)(4-|\Delta|^2)}}{2}.
\eeq

CASE 2: Suppose now that $\mathcal{H}<-|\Psi|^2$. This implies that $|\Sigma|^2+|\Delta|^2>4$, and we do not have a priori the bounds $|\Sigma|\leq 2$, $|\Delta|\leq 2$. Nevertheless, we derive equation \eqref{NEG} again, and we find that the values of $|\Psi|$ can be any of 
\beql{ps3}
|\Psi|_{1,2}=\frac{|\Sigma||\Delta|\pm\sqrt{(4-|\Sigma|^2)(4-|\Delta|^2)}}{2},
\eeq
provided that the roots are real, namely we have either $|\Sigma|>2$ and $|\Delta|>2$ or $|\Sigma|\leq 2$ and $|\Delta|\leq 2$. The first case is, however, not possible, as it would imply $|\Psi|>2$ contradicting the crucial condition \eqref{NI}.

Once we have established the bounds $|\Sigma|\leq 2$, $|\Delta|\leq 2$ and the value(s) of $|\Psi|$ has been found, we are free to use the Decomposition formula to obtain the values of $s_1$, $s_2$, $s_3$ and $s_4$. Clearly, we can set $s_1$ with the $+$ sign, while $s_2$ with the $-$ sign as in formula \eqref{s12}, up to equivalence. However, we do not know a priori how to distribute the signs amongst $s_3$ and $s_4$, and to simplify the notations we define
\beql{S34}
s_3=-\frac{\Delta}{2}\pm\mathbf{i}\frac{\Delta}{|\Delta|}\sqrt{1-\frac{|\Delta|^2}{4}},\ \ \ \ 
s_4=-\frac{\Delta}{2}\mp\mathbf{i}\frac{\Delta}{|\Delta|}\sqrt{1-\frac{|\Delta|^2}{4}}.
\eeq
In particular, by using \eqref{SH} and \eqref{S34} we find that the orthogonality equation \eqref{ORT} becomes
\beql{FIN}
\frac{4-|\Sigma|^2-|\Delta|^2-|\Psi|^2}{\Sigma\overline{\Delta}}+\frac{\overline{\Sigma}\Delta}{2}\pm2\frac{\overline{\Sigma}\Delta}{|\Sigma||\Delta|}\sqrt{1-\frac{|\Sigma|^2}{4}}\sqrt{1-\frac{|\Delta|^2}{4}}=0,
\eeq
where the $\pm$ sign agrees with the definition of $s_3$. To conclude the theorem plug in all of the possible values of $|\Psi|$ as described in \eqref{ps1}, \eqref{ps2}, \eqref{ps3} into \eqref{FIN} to verify that for some choice of the sign it holds identically.
\end{proof}
\begin{rem}
The two possible signs described by formula \eqref{ps3} can be realized. In particular, there are two different orthogonal triplet of rows composed of sixth roots of unity where $|\Sigma|=|\Delta|=\sqrt{3}$ in both cases, however, in one of the cases $|\Psi|=1$ while $|\Psi|=2$ in the other.
\end{rem}
\begin{corollary}[(Haagerup's trick, \cite{haagerup})]\label{HT}
Suppose that the rows $(1,1,1,1,1,1), (1,a,b,e,s_1,s_2)$ and $(1,c,d,f,s_3,s_4)$ composed of unimodular entries are mutually orthogonal. Then
\beql{HEQ}
\mathcal{H}(a,b,c,d,e,f)\in\mathbb{R}.
\eeq
\end{corollary}
Haagerup used the property \eqref{HEQ} to give a complete characterization of complex Hadamard matrices of order $5$, or equivalently, describe the orthogonal maximal abelian $\ast$-subalgebras of the $5\times 5$ matrices \cite{haagerup}. Since then it was used in \cite{BN} and \cite{MSZ} to construct new, previously unknown complex Hadamard matrices of order $6$ as well. However, to guarantee the mutual orthogonality of three rows the necessary condition \eqref{HEQ} should be replaced by the more stronger identity \eqref{C11}. Nevertheless, \eqref{HEQ} will play an essential r\^ole in this paper too. These type of identities are extremely useful as they feature less variables than the standard orthogonality equations considerably simplifying the calculations required.

Solving the system of equations \eqref{C11}--\eqref{HEQ} is the key step to obtain the submatrices $B$ and $C$ of $G_6$. Once we have three orthogonal rows and columns we readily fill out the remaining lower right submatrix  $D$. This is explained by the following two lemmata, the first of which being a special case of a more general matrix inversion
\begin{lemma}\label{MIL}
If $U$ and $V$ are $n\times n$ matrices then \[\left(I_n+UV\right)^{-1}=I_n-U\left(I_n+VU\right)^{-1}V\]
provided that one of the matrices $I_n+UV$ or $I_n+VU$ is nonsingular.
\end{lemma}
\begin{proof}
By symmetry, we can suppose that the matrix $I_n+VU$ is nonsingular. Then, we have
\[(I_n+UV)(I_n-U(I_n+VU)^{-1}V)=I_n+UV-U(I_n+VU)(I_n+VU)^{-1}V=I_n.\]
\end{proof}
\begin{lemma}\label{DDD}
Suppose that we have a $6\times 6$ partial complex Hadamard matrix consisting of three orthogonal rows and columns, containing no vanishing $3\times 3$ minor. Then there is a unique way to construct a unitary matrix containing these rows and columns as a submatrix.
\end{lemma}
\begin{proof}
Let $U$ be a $6\times 6$ matrix with $3\times 3$ blocks $A,B,C$ and $D$, as the following:
\[U=\left[\begin{array}{cc}
A & B \\
C & D\\
\end{array}\right].\]
By the orthogonality of the first three rows and columns and using the fact that the entries are unimodular, we have
\beql{28}
AA^\ast+BB^\ast=6I_3
\eeq
\beql{29}
A^\ast A+C^\ast C=6I_3
\eeq
To ensure orthogonality in-between the first three and the last three rows we need to have $AC^\ast+BD^\ast=O_3$, the all $0$ matrix. As $B$ is nonsingular by our assumptions we can define
\beql{DD}
D:=-CA^\ast(B^{-1})^\ast.
\eeq
Now we need to show that the last three rows are mutually orthogonal as well. Indeed, by using \eqref{DD} and \eqref{28} we have
\[CC^\ast+DD^\ast=C\left(I_3+A^\ast(BB^\ast)^{-1}A\right)C^\ast=C\left(I_3+A^\ast\left(6I_3-AA^\ast\right)^{-1}A\right)C^\ast,\]
which, by Lemma \ref{MIL} and \eqref{29} is
\[C\left(I_3+\frac{1}{6}A^\ast\left(I_3-\frac{1}{6}AA^\ast\right)^{-1}A\right)C^\ast=C\left(I_3-\frac{1}{6}A^\ast A\right)^{-1}C^\ast=6C\left(C^\ast C\right)^{-1}C^\ast=6I_3.\]
\end{proof}
We do not state that the obtained unitary matrix $U$ is Hadamard, which is not true in general. Recall that our goal is to embed the submatrix $E$ into the matrix $G_6$ (cf.\ \eqref{E}-\eqref{G6}). We have the following trivial
\begin{lemma}
Suppose that a submatrix $E$ can be embedded into a complex Hadamard matrix $G_6$ of order $6$ in which the upper right submatrix $B$ is invertible. Then for some unimodular submatrix $C$ for which the first three columns of $G_6$ are orthogonal the lower right submatrix $D=-CE^\ast(B^{-1})^\ast$ is unimodular.
\end{lemma}
\begin{proof}
Indeed, this is exactly what embedding means.
\end{proof}
In particular, if the submatrices $B$ and $C$ are chosen carefully, the unimodular property of $D$ follows for free.
\begin{corollary}\label{embC}
Start from a submatrix $E$ and suppose that there are only finitely many (invertible) candidate matrices $B\in SOL_B$ and $C\in SOL_C$ such that the first three rows and columns of the matrix $G_6$ are orthogonal. Then $E$ can be embedded into a complex Hadamard matrix of order $6$ if and only if there is some $B\in SOL_B$ and $C\in SOL_C$ such that the matrix $D=-CE^\ast(B^{-1})^\ast$ is unimodular.
\end{corollary}
Note that due to the finiteness condition in Corollary \ref{embC} once we have all (but finitely many) candidate matrices $B$ and $C$ we can decide algorithmically whether the submatrix $E$ can be embedded into a complex Hadamard matrix.

The next step is to characterize $6\times 6$ complex Hadamard matrices with vanishing $3\times 3$ minors. To do this we need two auxiliary results first.
\begin{lemma}\label{AUX1}
Suppose that in a dephased $6\times 6$ complex Hadamard matrix there exist a noninitial row (or column) containing three identical entries $x$. Then $x=\pm 1$ and this row (or column) reads $(1,1,1,-1,-1,-1)$, up to permutations.
\end{lemma}
\begin{proof}
Suppose to the contrary, that there are three nonreal numbers $x$ in a row (column). Then the sum of these numbers $x$ together with the leading $1$ read $|1+3x|^2=10+6\Re(x)>4$, and hence, by Lemma \ref{L1} this row (column) cannot be orthogonal to the first row (column), a contradiction. To ensure orthogonality to the first row (column) we should specify the remaining three entries to $-x$ and hence the last part of the statement follows.
\end{proof}
Recall that the core of a dephased complex Hadamard matrix of order $n$ is its lower right $(n-1)\times (n-1)$ submatrix. A vanishing sum of order $k$ is a $k$-term sum adding up to $0$. The following breakthrough result was obtained very recently.
\begin{theorem}[(Karlsson, \cite{K1,K2})]\label{K6}
Let $H$ be a dephased complex Hadamard matrix of order $6$. Then the following are equivalent:
\begin{enumerate}[(a)]
\item $H$ belongs to the three-parameter degenerate family $K_6^{(3)}$;
\item $H$ is $H_2$-reducible;
\item The core of $H$ contains a $-1$;
\item Some row or column of $H$ contains a vanishing sum of order $2$;
\item Some row or column of $H$ contains a vanishing sum of order $4$.
\end{enumerate}
\end{theorem}
In particular, Theorem \ref{K6} gives a characterization of complex Hadamard matrices of order $6$ containing $F_2$ as a submatrix. The term $H_2$-reducibility refers to the beautiful structure of these matrices: they have a canonical form in which all $9$ of their $2\times 2$ submatrices are complex Hadamard. Part (c) and (d) of Theorem \ref{K6} allow us to quickly recognize if a matrix belongs to the family $K_6^{(3)}$, and we shall heavily use these conditions through our paper. Part (e) tells us that once we have four entries in a row or column of a matrix which lies outside the family $K_6^{(3)}$ the Decomposition formula readily derives the unique remaining two values through \eqref{s12}. As the family $K_6^{(3)}$ forms a three-parameter subset, the matrices it contains are atypical, hence the adjective ``degenerate''.
\begin{corollary}\label{C213}
Suppose that in a dephased $6\times 6$ complex Hadamard matrix $H$ there exist a noninitial row (or column) containing three identical entries. Then $H$ belongs to the family $K_6^{(3)}$.
\end{corollary}
\begin{lemma}\label{Vanish}
Suppose that a $6\times 6$ complex Hadamard matrix $H$ has a vanishing $3\times 3$ minor. Then $H$ belongs to the family $K_6^{(3)}$.
\end{lemma}
\begin{proof}
Suppose that $H$ has a vanishing $3\times 3$ minor, say the upper left submatrix $E(a,b,c,d)$, as in formula \eqref{E}. Such an assumption can be made, up to equivalence. As $\mathrm{det}(E)=b+c-a-d+ad-bc=0$, we find that if any of the indeterminates $a,b,c,d$ is equal to $1$ then $E$ contains a noninitial row (or column) containing three $1$s, and therefore by Corollary \ref{C213} we conclude that this matrix belongs to the family $K_6^{(3)}$. Otherwise, we can suppose that none of $a,b,c,d$ is equal to $1$ and hence we find that $d=(a+b c-b-c)/(a-1)$, which should be of modulus one. To ensure this solve the equation $d\overline{d}-1=0$ to find that either $b=a$ or $c=a$ should hold, but then we have either $d=c$ or $d=b$ as well. In particular, we find that $E$ has two identical rows (or columns). After enphasing the matrix, again, we find that there is a full column (or row) of entries $1$ and a reference to Corollary \ref{C213} concludes the lemma.
\end{proof}
Therefore to investigate those matrices which lie outside the family $K_6^{(3)}$ we can safely use Lemma \ref{DDD} and in particular the inversion formula \eqref{DD}.

It turns out, that the isolated matrix $S_6^{(0)}$ (cf.\ \cite{karol}) requires a special treatment as well. It is featured in the following
\begin{lemma}\label{cubic}
Suppose that in a $6\times 6$ dephased complex Hadamard matrix $H$ there is a noninitial row and column composed of cubic roots of unity. Then $H$ is either equivalent to $S_6^{(0)}$ or belongs to the family $K_6^{(3)}$.
\end{lemma}
Let us denote by $\omega$ the principal cubic root of unity once and for all, that is $\omega:=\mathbf{e}^{2\pi\mathbf{i}/3}$.
\begin{proof}
First suppose that the cubic row and column meet in a common $1$. Then our matrix looks like as the matrix $H$ on the left below, up to equivalence:
\[H=\left[\begin{array}{rrrrrr}
1 & 1 & 1 & 1 & 1 & 1\\
1 & 1 & \omega & \omega & \omega^2 & \omega^2\\
1 & \omega & a & b & c & d\\
1 & \omega & \ast & \ast & \ast & \ast\\
1 & \omega^2 & \ast & \ast & \ast & \ast\\
1 & \omega^2 & \ast & \ast & \ast & \ast\\
\end{array}\right],\ \ \ 
H'=\left[\begin{array}{rrrrrr}
1 & 1 & 1 & 1 & 1 & 1\\
1 & 1 & \omega & \omega & \omega^2 & \omega^2\\
1 & a & 1 & b & c & d\\
1 & \ast & \omega & \ast & \ast & \ast\\
1 & \ast & \omega^2 & \ast & \ast & \ast\\
1 & \ast & \omega^2 & \ast & \ast & \ast\\
\end{array}\right].\]
Now by orthogonality of the first three rows we find that $a+b=-\omega$ and $c+d=-1$. Hence, we can assume, up to equivalence, that $a=1, b=-\omega^2$, and $c=\omega$, $d=\omega^2$. But then, we can fill out the fourth row, and the third and fourth column as well. We conclude that the obtained matrix is equivalent to $S_6^{(0)}$.

Secondly, let us suppose that the cubic row and column meet in a common $\omega$. This matrix $H'$ is depicted on the right above. Then, by calculating the orthogonality equations, we find that either $a=-1$, and hence the matrix (if it can be completed to a Hadamard at all) belongs to the family $K_6^{(3)}$, or $a=\omega$, $b=\omega^2$, and, up to equivalence $c=\omega$, $d=\omega^2$. This implies that the third row and third column feature cubic entries only meeting in a common $1$ therefore reducing the situation to the first case. The third case, namely when the cubic row and column meet in a common $\omega^2$ can be treated similarly.
\end{proof}
Now we turn to the presetting of the submatrix $E(a,b,c,d)$ (see \eqref{E}). In order to avoid the case when the system of equations \eqref{C11}--\eqref{HEQ} is linearly dependent we need to exclude various input quadruples $(a,b,c,d)$. However, it shall turn out that we are free to do such restrictions, up to equivalence. Before proceeding further, let us define the following two-variable function mapping $\mathbb{T}^2$ to $\mathbb{C}$ as follows:
\[\mathcal{E}(x,y):=x+y+x^2+y^2+xy^2+x^2y.\]
We say that $y$ is an elliptical pair of $x$, if $\mathcal{E}(x,y)=0$. Observe that for a given $x\neq -1$ the sum of its elliptical pairs read $y_1+y_2=-(1+x^2)/(1+x)$. The following is a strictly technical
\begin{proposition}[(Canonical transformation)]\label{CTr}
Suppose that we have a complex Hadamard matrix $H$ inequivalent from $S_6^{(0)}$ and any of the members of the family $K_6^{(3)}$. Then $H$ has a $3\times 3$ submatrix $E(a,b,c,d)$ as in formula \eqref{E}, up to equivalence, satisfying
\beql{canCC}
(b-1)(c-1)(b-d^2)(c-d^2)(b-c)(bc-d)\mathcal{E}(b,d)\mathcal{E}(c,d)\neq 0.
\eeq
\end{proposition}
\begin{proof}
The strategy of the proof is the following: first we pick a ``central element'' $d$ from the core of the matrix and then we show that $b$ and $c$ can be set satisfying \eqref{canCC}. Recall, that by Lemma \ref{cubic} there is no a noninitial row and column composed from cubic roots of unity in the matrix.

First let us assume that there is a $1$ in the core somewhere. Suppose that there is a, to say, row full of cubics. In this case set $d=\omega$, and $c=\omega^2$. Now in the column containing $d$ there is a non-cubic entry $\gamma$, and we are free to set $b=\gamma$. If there is neither a full row nor a column of cubics in the matrix, then set $d=1$, and choose a non-cubic $c$ from its row. Note that there cannot be a further noninitial $1$ in the row or column of $d$ by Corollary \ref{C213}. Now observe that as the elliptical pairs of $1$ are $\omega$ and $\omega^2$, we can choose a suitable $b$ from the column of $d$ unless all entries there are members of the set $\{\omega, \omega^2,c,\overline{c}\}$. Note that $\omega$ together with $\omega^2$ cannot be in the column of $d$ at the same time, and from this it is easily seen that we cannot define the value of $b$ only if the column of $d$ is one of the following four cases, up to permutations: $(1,1,\omega,c,c,\overline{c})$, $(1,1,\omega,c,\overline{c},\overline{c})$, $(1,1,\omega^2,c,c,\overline{c})$, $(1,1,\omega^2,c,\overline{c},\overline{c})$. However, by normalization and by orthogonality, the sum of the entries in this column should add up to $0$, and we find in all cases that the unimodular solution to $c$ is a cubic root of unity, contradicting the choice of $c$. Therefore one of the entries in the column of $d$ is different from $\omega,\omega^2,c,\overline{c}$ which will be chosen as $b$.

Secondly, let us suppose, that there is no $1$ in the core. In particular, all entries in the core are different rowwise and columnwise.

Pick any $d$ from the core of the matrix. Let us denote by $c_1$ and $c_2$ the elliptical pairs of $d$ (maybe $c_1=c_2$, or they are undefined). Now we have several cases depending on the appearance of these values in the row and column containing $d$.

CASE $1$: $c_1$ and $c_2$ present in both the row and column containing $d$. Hence, in the row and column containing $d$ there are the entries $1$, $d$, $c_1$ and $c_2$ already, the remaining two, $\alpha$ and $\beta$ are uniquely determined by the Decomposition formula. Note that $\alpha\neq d^2$, as otherwise we would have $\beta=-(2d+d^2+d^3)/(1+d)$ (as the sum of all entries in a noninitial row add up to $0$, and the sum of the elliptical pairs is known), which is unimodular if and only if $d=\pm\mathbf{i}$ or $d=\omega$ or $d=\omega^2$. In the first case we have $\beta=\mp\mathbf{i}$ and we are dealing with a member of the family $K_6^{(3)}$. The second and third case imply that we have a full row and a full column of cubics, a contradiction. If $\beta\neq d/\alpha$ then by picking $c=\alpha$ from the row we can set $b=\beta$ in the column. Otherwise, in case we have $\beta=d/\alpha$, then reset the central element $d$ to the $\alpha$ which is in the same row as $d$. Now observe that after this exchange we should met the requirements of Case $1$ again (otherwise we are done here), and hence the elliptical pairs of $\alpha$, $\alpha_1$ and $\alpha_2$, should present in the row and column of $\alpha$, which therefore contain exactly the same entries. However $\alpha_{1,2}\neq d$, as otherwise orthogonality with the condition $\mathcal{E}(\alpha,d)=0$ would imply $d=\pm1$, a contradiction; $\alpha_{1,2}\neq d/\alpha$, as otherwise $d=\omega$ or $d=\omega^2$ would follow from the same argument implying that the row containing $\alpha$ has a noninitial $1$, a contradiction. Therefore, the only option left is that $\alpha_{1,2}=c_{1,2}$. But then we can set $c=d\neq\alpha^2$ and $b=d/\alpha\neq \alpha^2$, and we are done. 

CASE $2$: $c_1$ and $c_2$ present in (to say) the row containing $d$, but only one of these values (say $c_1$) is present in the column of $d$. Let us denote by $\alpha$ and $\beta$ the two further entries in this row which, again, are different from $d^2$. In the column of $d$ there is already $1$, $d$ and $c_1$, and observe that the remaining three entries cannot be $(d^2,\alpha,\beta)$ as this would imply $d^2=c_2$, contradicting our case-assumption. Therefore one of the three unspecified entries $\gamma$ is different from $d^2$, $\alpha$ and $\beta$. Now if $\gamma\neq d/\alpha$ then set $c=\alpha$, $b=\gamma$ otherwise set $c=\beta$, $b=\gamma$. We are done.

CASE $3$: Only the value $c_1$ is in (to say) the row containing $d$ and in the column of $d$ as well. This is a tricky case, as it might happen that the undetermined triplet in both the $d$-th row and column is precisely $(d^2, \alpha, d/\alpha)$, and therefore we cannot ensure condition \eqref{canCC}. However, from the orthogonality equation $1+d+d^2+c_1+\alpha+d/\alpha=0$, from its conjugate, and the elliptical condition $\mathcal{E}(c_1,d)=0$ we can form a system of equations, the solution of which can be found by computing a Gr\"obner basis. By investigating the results, we find that either $\alpha=-1$ or $\alpha+d=0$ leading us to the family $K_6^{(3)}$ or $c_1=\overline{d}$. In the last case, however, we can calculate the values of $d$ and $\alpha$ explicitly. In particular, we find that the values of $d$ and $\alpha$ are given by some of the unimodular roots of the following polynomials $1+2d+2d^3+d^4=0$ and $1+4\alpha-2\alpha^2-8\alpha^3-8\alpha^4-8\alpha^5-2\alpha^6+4\alpha^7+\alpha^8=0$.
Now reset the ``central'' entry $d$ to $\alpha$, and observe that the elliptical pairs of $\alpha$ are not present in its row, and therefore the conditions of this subcase are no longer met. Otherwise we can suppose that the undetermined triplet in the column of $d$ is not $(d^2,\alpha,d/\alpha)$. Set $c=\alpha\neq d^2$. Pick $\gamma$ from the column which is different from $d^2$, $\alpha$, $d/\alpha$, set $b=\gamma$ and we are done.

CASE $4$: Only the value $c_1$ is in (to say) the row containing $d$ and in the column of $d$ there is the other elliptical value $c_2\neq c_1$. We can suppose that two of the undetermined entries in the row of $d$ satisfy $\alpha\neq d^2$ and $\beta\neq d^2$ and set $c=\alpha$. Now if in the column the undefined triplet is precisely $(\alpha$, $d/\alpha$, $d^2)$, then observe that the same triplet cannot appear in the row, as otherwise $c_1=c_2$ would follow. Therefore we can reset $c$ to a value different from $\alpha, d/\alpha, d^2$, and set $b=\alpha$. Otherwise there is an entry in the column which can be set to $b$, we are done.

CASE $5$: In the column of $d$ there is no elliptical value at all. Pick any $c=\alpha\neq d^2$ from the row. Now in the column there are four unspecified entries. Clearly, one of them, say $\gamma$ will be different from $d^2$, $\alpha$ and $d/\alpha$. Set $b=\gamma$. We are done.
\end{proof}
Not every submatrix $E$ can be embedded into a complex Hadamard matrix of order $6$. To offer a necessary condition, let us recall first that an operator $A$ is called a contraction, if $\left\|A\right\|_2\leq 1$, where $\left\|.\right\|_2$ denotes both the Euclidean norm on $\mathbb{C}^6$ and the induced operator norm on the space of $6\times 6$ matrices. We have the following
\begin{lemma}\label{Eig}
If $A$ is any $3\times 3$ submatrix of a complex Hadamard matrix $H$ of order $6$ then $A/\sqrt{6}$ is a contraction.
\end{lemma}
\begin{proof}
Clearly, we can assume that this submatrix $A$ is the upper left of the matrix $H$, which we will write in block form, as follows:
\[H=\left[\begin{array}{cc}
A & B\\
C & D\\
\end{array}\right].\]
Now suppose, to the contrary that there is some vector $s$, such that $\left\|As\right\|_2>\sqrt6\left\|s\right\|_2$ and consider the block vector $s':=(s,0)^T\in\mathbb{C}^6$. We have
\[\left\|Hs'\right\|_2=\left\|(As,Cs)^T\right\|_2\geq\left\|(As,0)^T\right\|_2=\left\|As\right\|_2>\sqrt6\left\|s\right\|_2=\sqrt6\left\|s'\right\|_2=\left\|Hs'\right\|_2,\]
where in the last step we used that the matrix $H/\sqrt6$ is unitary.
\end{proof}
In particular, we have the following
\begin{corollary}\label{CE}
If the submatrix $E$ can be embedded into a complex Hadamard matrix of order $6$, then every eigenvalue $\lambda$ of the matrix $E^\ast E$ satisfy $\lambda\leq 6$.
\end{corollary}
Corollary \ref{CE} is a useful criterion to show that a matrix $E$ cannot be embedded into a complex Hadamard matrix, however, it is unclear how to utilize it for our purposes. In particular, we do not know how to characterize those $3\times 3$ matrices which satisfy its conditions. Also it is natural to ask whether the presence of the large eigenvalues is the only obstruction forbidding the submatrix $E$ to be embedded. The answer to this question might depend on the dimension, as it is easily seen that while every $2\times 2$ matrix can be embedded into a complex Hadamard matrix of order $4$, only a handful of very special $2\times 2$ matrices can be embedded into a complex Hadamard matrix of order $5$ due to the finiteness result of Haagerup \cite{haagerup}.

Now we are ready to present a new, previously unknown family of complex Hadamard matrices. The next section gives an overview of the results.
\section{The construction: A high-level perspective}\label{HLev}
Here we describe the generic family $G_6^{(4)}$ from a high-level perspective. In particular, we outline the main steps only, and do not discuss some degenerate cases which might come up during the construction. The next section is dedicated to investigate the process in details. The main result of this paper is the following
\begin{construction}[(The Dilation Algorithm)]\label{mC}
Do the following step by step to obtain complex Hadamard matrices of order $6$.\normalfont
\begin{enumerate}[\#1: ]
\item \ttfamily{INPUT}\normalfont: the quadruple $(a,b,c,d)$, forming the upper left $3\times 3$ submatrix $E(a,b,c,d)$, as in formula \eqref{E}.
\item Use Haagerup's trick to the first three rows of $G_6^{(4)}$ (see \eqref{G6}) to obtain a quadratic equation to $f$:
\beql{Flin}
\mathcal{F}_1+\mathcal{F}_2f+\mathcal{F}_3f^2=0,
\eeq
where the coefficients $\mathcal{F}_1$, $\mathcal{F}_2$ and $\mathcal{F}_3$ depend on the parameters $a,b,c,d$ and the indeterminate $e$, and derive the following linearization formula from it:
\beql{lin}
f^2=-\frac{\mathcal{F}_1}{\mathcal{F}_3}-\frac{\mathcal{F}_2}{\mathcal{F}_3}f.
\eeq
\item Use Theorem \ref{HTIO} to obtain another quadratic equation to $f$:
\beql{Glin}
\mathcal{G}_1+\mathcal{G}_2f+\mathcal{G}_3f^2=0,
\eeq
where, again, the coefficients $\mathcal{G}_1$, $\mathcal{G}_2$ and $\mathcal{G}_3$ depend on the parameters $a,b,c,d$ and the indeterminate $e$. Plug the linearization formula \eqref{lin} into \eqref{Glin} and rearrange to obtain the companion value of $e$ $f=F(e)$, where
\beql{formF}
F(e):=-\frac{\mathcal{F}_3\mathcal{G}_1-\mathcal{F}_1\mathcal{G}_3}{\mathcal{F}_3\mathcal{G}_2-\mathcal{F}_2\mathcal{G}_3}.
\eeq
\item As $|f|=1$ should hold, calculate the sextic polynomial $G(e)$ coming from the equation $F(e)\overline{F(e)}-1= 0$ and solve it for $e$.
\item Amongst the roots of $G$ find all unimodular triplets $(e,s_1,s_2)$ satisfying $e+s_1+s_2=-1-a-b$, calculate the companion values $f=F(e)$, $s_3=F(s_1)$ and $s_4=F(s_2)$ through formula \eqref{formF} and store all sextuples $(e,s_1,s_2,f,s_3,s_4)$ in a solution set called $SOL_B$.
\item Repeat steps $\#2$--$\#6$ to the transposed matrix (i.\ e.\ to the first three columns), mutatis mutandis to obtain the solution set $SOL_C$.
\item For every pair of sextuples from $SOL_B$ and $SOL_C$ construct the submatrices $B$ and $C$, check if the first three rows and columns are mutually orthogonal and finally use Lemma \ref{DDD} to compute the lower right submatrix $D$ through formula \eqref{DD}.
\item \ttfamily{OUTPUT:}\normalfont all unimodular matrices found in step $\#7$.
\end{enumerate}
\end{construction}
Construction \ref{mC} gives the essence of the new family discovered, and in the next section we shall give it a mathematically rigorous, low-level look.
\section{The construction: The nasty details}
Here we investigate the steps of Construction \ref{mC} in details.

\smallskip
\noindent STEP \#1: Choose a quadruple $(a,b,c,d)$ in compliance with the Canonical Transformation described by Proposition \ref{CTr} as the \ttfamily{INPUT}\normalfont, and form the submatrix $E$. Check if it meets the requirements of Corollary \ref{CE}; if yes, then proceed, otherwise conclude that it cannot be embedded into a complex Hadamard matrix of order $6$. Experimental results show that once three out of the four parameters are fixed the last one can be easily set to a value such that the quadruple $(a,b,c,d)$ leads to a complex Hadamard matrix. Heuristically this means that there is nothing ``mystical'' in the choice of the initial quadruple and hence the parameters should be independent from each other.

\smallskip
\noindent STEP \#2: To obtain formula \eqref{lin} we need to see that $\mathcal{F}_3\not\equiv 0$, independently of $e$. Indeed, suppose otherwise, which means that the following system of equations (where the last two are the conjugate of the first two, up to some irrelevant constant factors)
\[\left\{\begin{array}{ccc}
a b c + a^2 b c + a b d + a b^2 d + a^2 b c d + a b^2 c d & \equiv & 0,\\
b^2 c + a b^2 c + a^2 d + a^2 b d + a c d + b c d & \equiv & 0,\\
a + b + a c + a b c + b d + a b d & \equiv & 0,\\
a^2 b + a b^2 + b c + b^2 c + a d + a^2 d & \equiv & 0,\\
\end{array}\right.\]
are fulfilled. We compute a Gr\"obner basis and find that the polynomial
\[b c (1 + c^2) (c - d) d (1 + d^2) (c^2 + d^2) (1 + d + d^2)\]
is a member of it. After substituting back into the original equations we find that there is either a vanishing sum of order $2$ in $E$ or $a=b=1$ and therefore the whole family is a member of $K_6^{(3)}$, or we have $E=F_3$ or $E=F_3^\ast$ but these matrices have $b=c$ which however is not allowed by the Canonical Transformation.

It might happen that $\mathcal{F}_3\not\equiv 0$ but there is a unimodular $e$ making it vanish, which cannot be anything else, but
\beql{FE}
e=\frac{a^2 b+a^2 d+a b^2+a d+b^2 c+b c}{a b c+a b d+a c+a+b d+b}.
\eeq
Nevertheless, we can suppose that in case of one of the pairs $(e,f)$, $(s_1,s_3)$ and $(s_2,s_4)$ we do not set $e$ as above, otherwise we would have $e=s_1=s_2$ which, by Lemma \ref{AUX1} would imply $e=s_1=s_2=-1$, obtaining some member of the family $K_6^{(3)}$ by Corollary \ref{C213}. Hence, we can suppose that $e$ is different than the value described by formula \eqref{FE} above, and we conclude that $\mathcal{F}_3\neq 0$.

\smallskip
\noindent STEP \#3:
Clearly, one cannot expect to recover a unique $f$ from $e$ in general, as formula \eqref{formF} might suggests. Indeed, there are complex Hadamard matrices in which $s_1=e$, but $s_3\neq f$. The reason for this phenomenon is that formulas \eqref{Flin} and \eqref{Glin} might be linearly dependent. After plugging \eqref{lin} into \eqref{Glin} we obtain the expression
\[\mathcal{F}_3\mathcal{G}_1-\mathcal{F}_1\mathcal{G}_3+\left(\mathcal{F}_3\mathcal{G}_2-\mathcal{F}_2\mathcal{G}_3\right)f=0,\]
which can lead us to one of the following three cases:

CASE 1: Both the polynomials $\mathcal{F}_3\mathcal{G}_1-\mathcal{F}_1\mathcal{G}_3$ and $\mathcal{F}_3\mathcal{G}_2-\mathcal{F}_2\mathcal{G}_3$ vanish identically, independently of $e$, meaning that in this case we do not have another condition on $f$. Luckily this can never happen, as by calculating a Gr\"obner basis (again, to speed up the computations we have added the conjugates of the equations as well) we find that the polynomial
\[(b-1) b^2 (b - c) c (b c - d) (1 + c + d) (b - d^2) \mathcal{E}(c,d)\]
is a member of the basis. Therefore we have either $c+d=-1$ or one of the degenerate cases described in the Canonical Transformation. The case $c+d=-1$ implies that the set $\{c,d\}$ consists of nontrivial cubic roots only. By directly solving the corresponding equations, we find that the quadruples $(1,1,\omega,\omega^2)$, $(\omega^2,\omega,\omega,\omega^2)$ and $(1,1,\omega^2,\omega)$, $(\omega,\omega^2,\omega^2,\omega)$ can vanish both polynomials, however these cases were excluded by the Canonical Transformation.

CASE 2: Both the polynomials $\mathcal{F}_3\mathcal{G}_1-\mathcal{F}_1\mathcal{G}_3$ and $\mathcal{F}_3\mathcal{G}_2-\mathcal{F}_2\mathcal{G}_3$ vanish for some $|e|=1$, meaning that in this case we do not have another condition on $f$. However, having these numbers $e$ at our disposal we can recover the two possible values of $f$ from \eqref{lin}. Once we have the candidate pairs $(e,f_1)$ and $(e,f_2)$ we readily calculate the remaining pairs $(s_1,s_3)$ and $(s_2,s_4)$ through the Decomposition Formula. Store all suitable sextics $(e,s_1,s_2,f,s_3,s_4)$ in the solution set $SOL_B$. Proceed to step $\#6$.

CASE 3: One cannot set (or have not set in Case $2$) a unimodular $e$ to make these two polynomial vanish at the same time. Hence we can derive formula \eqref{formF} for $F(e)$. Proceed to step $\#4$.

\smallskip
\noindent STEP \#4:
Next we need to ensure that $f$ is of modulus one. To do this, we calculate the fundamental polynomial
\[\mathcal{P}_{a,b,c,d}(e)\equiv \left|\mathcal{F}_3\mathcal{G}_1-\mathcal{F}_1\mathcal{G}_3\right|^2-\left|\mathcal{F}_3\mathcal{G}_2-\mathcal{F}_2\mathcal{G}_3\right|^2.\]
After some calculations, it will be apparent that $\mathcal{P}$ has the following remarkable structure:
\[\mathcal{P}\equiv a^8b^8c^6d^6\overline{P}+a^8b^8c^6d^6\overline{P}(2+2\overline{a}+2\overline{b})e+a^8b^8c^6d^6\overline{Q}e^2+Re^3+Qe^4+P(2+2a+2b)e^5+Pe^6\]
where the coefficients $P, Q$ and $R$ depend on the quadruple $(a,b,c,d)$ only. If $\mathcal{P}\equiv 0$ then the construction fails. Otherwise we find all possible roots of $\mathcal{P}(e)$ of modulus one. 

\smallskip
\noindent STEP \#5:
If the number on the right hand side of \eqref{FE} is not of modulus one, then the r\^ole of the pairs $(e,f), (s_1,s_3)$ and $(s_2,s_4)$ is symmetric, and from all of the roots of $\mathcal{P}$ of modulus one we select all possible triplets $(e,s_1,s_2)$ satisfying $e+s_1+s_2=-1-a-b$, which is needed to ensure orthogonality of the first two rows. From \eqref{formF} we compute the unique companion values $f=F(e), s_3=F(s_1)$ and $s_4=F(s_2)$.

Otherwise, should the number on the right hand side of \eqref{FE} is of modulus one, then for every root $e$ of $\mathcal{P}$ we calculate its unique companion value $f=F(e)$, and then we use the Decomposition formula to determine the pairs $(s_1,s_3)$ and $(s_2,s_4)$.

If no unimodular roots are found at this point then the matrix $E(a,b,c,d)$ cannot be embedded into a complex Hadamard matrix of order $6$, and we do not proceed any further.

At the end of this step we store all obtained sextics $(e,s_1,s_2,f,s_3,s_4)$ in the solution set $SOL_B$. Typically two sextics are found.

\smallskip
\noindent STEP \#6:
In this step one constructs the first three columns along the lines of steps $\#2$--$\#5$ described above and obtain the set $SOL_C$ in a similar way.

\smallskip
\noindent STEP \#7:
For every candidate solution from $SOL_B$ and $SOL_C$ we check if the first three rows and columns are orthogonal, disregard those cases in which the submatrix $B$ is singular and finally use Lemma \ref{DDD} to obtain the lower right submatrix $D$. Note that by Lemma \ref{Vanish} we disregard members of the family $K_6^{(3)}$ only. We check if $D$ is composed of unimodular entries.

\smallskip
\noindent STEP \#8:
Finally, we \ttfamily OUTPUT \normalfont all unimodular matrices found during the process. We remark here that by Corollary \ref{embC} if no unimodular matrices were found then the submatrix $E$ cannot be embedded into any complex Hadamard matrices of order $6$. If unimodular matrices are found, then typically we find two matrices, as the solution set $SOL_B$ and $SOL_C$ contains two suitable sextics each, however, experimental results show that for each sextic in $SOL_B$ there is a unique sextic in $SOL_C$ making $D$ unimodular as required.

\bigskip
We have finished the discussion of Construction \ref{mC}. The results are summarized in the following
\begin{theorem}
Start from a submatrix $E$ as in \eqref{E} and suppose that there are only finitely many (invertible) candidate submatrices $B$ and $C$ such that the first three rows and columns of the matrix $G_6$ (see \eqref{G6}) are orthogonal. Then Construction \ref{mC} gives an exhaustive list of all complex Hadamard matrices of order $6$, up to equivalence, containing $E$ as a submatrix.
\end{theorem}
The interested reader might want to see an example of generic Hadamard matrices which can be described by closed analytic formul\ae, that is for which the fundamental polynomials $\mathcal{P}_{a,b,c,d}$ and $\mathcal{P}_{a,c,b,d}$ are both solvable. Such a matrix can be obtained when we choose the input quadruple $(a,\overline{a},c,a)$ where the real part of $a$ is the unique real solution of $4\Re[a]^3-2\Re[a]+1=0$ and $c=(-a^3+a^2+a+1)/(a^4+a^3+a^2-a)$.
It is easily seen that the matrices we obtain starting from the submatrix $E(a,\overline{a},c,a)$ are inequivalent from $S_6^{(0)}$ and do not belong to the family $K_6^{(3)}$.
\begin{rem}
When $\mathcal{P}\equiv 0$ then the main difficulty we are facing with is that we have infinitely many candidate submatrices $B$. In this case we have the trivial restriction $\eqref{eq1}$ on $e$, while the companion value $f$ coming from \eqref{formF} is unimodular unconditionally. Although in principle we can find three orthogonal rows through the Decomposition formula for every suitable $e$, we do not know which one to favourize in order to obtain a unimodular submatrix $D$ via formula \eqref{DD}. Also, it might happen that the polynomial $\mathcal{P}_{a,c,b,d}$ obtained during step $\#6$ shall vanish as well bringing another free parameter into the game making things even more complicated. In contrast, if both $\mathcal{P}_{a,b,c,d}\not\equiv0$ and $\mathcal{P}_{a,c,b,d}\not\equiv0$, then we have a finitely many choices for the submatrices $B$ and $C$ and we can use Corollary \ref{embC} to conclude the construction.
\end{rem}
\begin{rem}
The polynomial $\mathcal{P}$ formally can vanish when we have $\mathcal{F}_3\mathcal{G}_1-\mathcal{F}_1\mathcal{G}_3\equiv\mathcal{F}_3\mathcal{G}_2-\mathcal{F}_2\mathcal{G}_3\equiv0$, however this is excluded by the Canonical Transformation and explained in details in Case $1$ of step $\#3$. It might vanish for some other, non-trivial quadruples as well making the whole construction process fail. In theory, the common roots of the coefficients of $\mathcal{P}$ can be calculated by means of Gr\"obner bases, but as these coefficients are rather complicated obtaining such a basis turned out to be a task beyond our capabilities. Nevertheless, we conjecture that the case $\mathcal{P}\equiv 0$ can be excluded completely in a similar fashion as we disregarded various quadruples during the Canonical Transformation. This would mean that all complex Hadamard matrices of order $6$, except from $S_6^{(0)}$ and $K_6^{(3)}$, can be recovered from Construction \ref{mC}.
\end{rem}
It is reasonable to think that every complex Hadamard matrix of order $6$ has some $3\times 3$ submatrix $E$ leading to nonvanishing fundamental polynomials. In particular, we do not expect any complex Hadamard matrices of order $6$ (except maybe $S_6^{(0)}$ and $K_6^{(3)}$) which cannot be recovered from Construction \ref{mC}. Therefore we formulate the following
\begin{conjecture}\label{C2}
The list of complex Hadamard matrices of order $6$ is as follows: the isolated matrix $S_6^{(0)}$, the three-parameter degenerate family $K_6^{(3)}$ and the four-parameter generic family $G_6^{(4)}$ as described above.
\end{conjecture}
It would be nice to understand the structure of $G_6^{(4)}$ more thoroughly and express the entries of these matrices by some well-chosen trigonometric functions in a similar fashion as $K_6^{(3)}$ is described, however, as we have encountered sextic polynomials already the appearance of such formulas is somewhat unexpected. Also, it is natural to ask which matrices satisfy the conditions of Corollary \ref{CE}. An algebraic characterization of these matrices might lead to a deeper understanding of the generic family $G_6^{(4)}$ and hopefully to the desired full classification of complex Hadamard matrices of order $6$.

\end{document}